\documentclass[11pt,reqno]{amsart}

\usepackage[T1]{fontenc}
\usepackage{lmodern}
\usepackage{microtype}
\usepackage{amsmath,amssymb,mathtools,amscd}
\usepackage{enumitem}
\usepackage{array}
\usepackage{booktabs}
\usepackage{aliascnt}
\usepackage{hyperref}
\usepackage[nameinlink,capitalize,noabbrev]{cleveref}

\hypersetup{
  colorlinks=true,
  linkcolor=blue,
  citecolor=blue,
  urlcolor=blue,
  pdftitle={Flat Projective Dimensions and n-Perfectness under Trivial and Split Nilpotent Extensions: A Corrected and Audited Version},
  pdfauthor={Samir Boutghouchout, Hwankoo Kim, Najib Mahdou}
}

\setlist[itemize]{leftmargin=2em}
\setlist[enumerate]{leftmargin=2.5em}

\newcommand{\pd}{\operatorname{pd}}
\newcommand{\fd}{\operatorname{fd}}
\newcommand{\cd}{\operatorname{cd}}
\newcommand{\cotD}{\operatorname{cot.D}}
\newcommand{\Tor}{\operatorname{Tor}}
\newcommand{\Ext}{\operatorname{Ext}}
\newcommand{\Spec}{\operatorname{Spec}}
\newcommand{\Mod}{\operatorname{Mod}}

\newcommand{\Coker}{\operatorname{Coker}}
\newcommand{\Ker}{\operatorname{Ker}}
\newcommand{\splf}{\operatorname{splf}}

\newcommand{\FlatPD}{\operatorname{FlatPD}}
\newcommand{\id}{\operatorname{id}}

\theoremstyle{plain}
\newtheorem{theorem}{Theorem}[section]
\newaliascnt{proposition}{theorem}
\newtheorem{proposition}[proposition]{Proposition}
\aliascntresetthe{proposition}
\newaliascnt{lemma}{theorem}
\newtheorem{lemma}[lemma]{Lemma}
\aliascntresetthe{lemma}
\newaliascnt{corollary}{theorem}
\newtheorem{corollary}[corollary]{Corollary}
\aliascntresetthe{corollary}

\theoremstyle{definition}
\newaliascnt{definition}{theorem}
\newtheorem{definition}[definition]{Definition}
\aliascntresetthe{definition}
\newaliascnt{example}{theorem}
\newtheorem{example}[example]{Example}
\aliascntresetthe{example}

\theoremstyle{remark}
\newaliascnt{remark}{theorem}
\newtheorem{remark}[remark]{Remark}
\aliascntresetthe{remark}
\newaliascnt{question}{theorem}

\aliascntresetthe{question}

\title[On the $n$-Perfectness under Split Nilpotent Extensions]
{On the $n$-Perfectness under Split Nilpotent Extensions}

\author{Wei Qi}
\address{School of Mathematics and Statistics, Tianshui Normal University, Tianshui 741001, China}
\email{qwrghj@126.com}

\author{Xiaolei Zhang}
\address{School of Mathematics and Statistics, Tianshui Normal University, Tianshui 741001, China}
\email{zxlrghj@163.com}


\begin{document}
\raggedbottom

\begin{abstract}
Let $\pi:S\twoheadrightarrow R$ be a split ring epimorphism with nilpotent kernel
$J$, and let $R\hookrightarrow S$ be a fixed ring-theoretic section.  We prove a
change-of-rings formula which is considerably stronger than an equality of global
invariants: if $X$ is a left $S$-module satisfying
$\Tor_i^S(R,X)=0$ for every $i>0$, then
\[
   \pd_S X=\pd_R(R\otimes_S X)
\]
in $\mathbb N\cup\{0,\infty\}$.  In particular, for every flat left
$S$-module $F$,
\[
   \pd_SF=\pd_R(F/JF),
\]
and $F/JF$ is flat over $R$.  Reduction and induction of scalars therefore show that the same numerical set of projective dimensions is attained by flat modules over $S$ and over $R$.  Consequently, left $n$-perfectness is invariant under every
split-by-nilpotent extension.  The right-handed analogue holds as well.

For a commutative ring $R$ and an arbitrary $R$-module $M$, this applies to the
idealization $S=R\ltimes M$ without any projectivity, flatness, finite generation,
or finite flat-dimension assumption on $M$.  We obtain a list of equivalent
characterizations of the condition that $R\ltimes M$ is $n$-perfect, together with
\[
   \cotD(R\ltimes M)=\cotD(R).
\]
We further derive exact preservation of the flat-projective-dimension spectrum, strict $n$-perfectness, pointwise localization, arbitrary iterated idealizations, $n$-trivial extensions, truncated polynomial rings, and upper triangular matrix rings.
\end{abstract}

\maketitle
\enlargethispage{4pt}

\section{Introduction}

A ring $A$ is called left $n$-perfect if every flat left $A$-module has
projective dimension at most $n$.  In the commutative setting this notion is
closely tied to cotorsion dimension: the global cotorsion dimension is the supremum
of the projective dimensions of flat modules.  The modern theory of flat cotorsion
modules was developed in particular by Enochs \cite{Enochs1984}, while cotorsion
dimension was studied systematically by Ding--Mao and by Bennis--Mahdou; see
\cite{MaoDing2006,BennisMahdou2009,BennisMahdouGor2009}.  For $n=0$ one recovers
Bass perfect rings, that is, rings over which every flat module is projective
\cite{Bass1960}.

Trivial ring extensions, also called idealizations, are among the most useful
split extensions in commutative and homological algebra.  Given a commutative ring
$R$ and an $R$-module $M$, the trivial extension
\[
   S=R\ltimes M
\]
has underlying abelian group $R\oplus M$ and multiplication
\[
   (r,m)(r',m')=(rr',rm'+r'm).
\]
The ideal $I=0\ltimes M$ satisfies $I^2=0$, the quotient $S/I$ is $R$, and the
quotient map has the canonical ring section $r\mapsto(r,0)$.  Trivial extensions
were studied in depth by Fossum, Griffith and Reiten
\cite{FossumGriffithReiten1975}; explicit homological-dimension formulae were
obtained by Palm\'er and Roos \cite{PalmerRoos1973}.  More recent structural and
categorical developments include \cite{AndersonEtAl2017,Beligiannis2000,
BenkhadraBennisGarcia2020}.  Split-by-nilpotent extensions also occur naturally in representation theory; see, for example, \cite{AssemMarmaridis1998,AssemZacharia2003}.  We use standard homological-algebra conventions as in
\cite{Rotman2009,Weibel1994}.

The projective idealizing case $R\ltimes P$, with $P$ projective over $R$, suggests
that cotorsion dimension and $n$-perfectness should be preserved.  The natural
question is whether projectivity can be dropped.  A first attempt might replace
projectivity by a bound $\fd_RM\le d$.  This, however, is not the correct level of
generality: the main theorem of the present paper shows that no homological
hypothesis on $M$ is needed at all.

There is a second reason to reorganize the proof.  The $R$-module decomposition
$S\cong R\oplus M$ is useful, but it is dangerous to treat decompositions obtained
from it as decompositions of $S$-modules.  For example,
\[
  \operatorname{Hom}_R(S,X)
  \cong X\oplus\operatorname{Hom}_R(M,X)
\]
is generally only an isomorphism of $R$-modules.  Applying $\Ext_S$ as though the
right side were an $S$-module direct sum can invalidate a change-of-rings argument.
Similarly, if $M$ is not flat, a flat $S$-module need not remain flat after
restriction to $R$.  These issues force one to identify the correct $R$-module
attached to a flat $S$-module.

The key object is not the restricted module but the reduction
\[
  R\otimes_SF\cong F/JF.
\]
In fact, we prove a theorem in the more general setting of a split-by-nilpotent
extension.  Let
\[
  \pi:S\twoheadrightarrow R,
  \qquad \iota:R\hookrightarrow S,
  \qquad \pi\iota=\id_R,
\]
and let $J=\ker\pi$ be nilpotent.  If an $S$-module $X$ is Tor-independent from
$R$ in the sense that
\[
  \Tor_i^S(R,X)=0\qquad(i>0),
\]
then
\[
  \pd_SX=\pd_R(R\otimes_SX).
\]
Flat $S$-modules automatically satisfy the Tor-vanishing condition.  This yields,
for every flat $S$-module $F$,
\[
  \pd_SF=\pd_R(F/JF),
\]
while $F/JF$ is flat over $R$.  Conversely, if $C$ is flat over $R$, then
$S\otimes_RC$ is flat over $S$ and has the same projective dimension as $C$.
Therefore the entire flat projective-dimension spectrum is preserved, not merely
its supremum.

For ordinary idealizations this gives the requested complete characterization:
for every $n\ge0$ and every $R$-module $M$,
\[
  R\ltimes M\text{ is }n\text{-perfect}
  \quad\Longleftrightarrow\quad
  R\text{ is }n\text{-perfect}.
\]
In the commutative case this is equivalent to
\[
  \cotD(R\ltimes M)=\cotD(R).
\]
The $n$-perfect equivalence for the broader class of $n$-trivial extensions was
obtained categorically by Benkhadra, Bennis and Garc\'ia Rozas
\cite[Theorem~4.2]{BenkhadraBennisGarcia2020}.  Our approach gives a modulewise strengthening in the present setting: it computes the projective dimension of every flat module, and more generally of every module satisfying the indicated Tor vanishing.  The proof is self-contained and applies to arbitrary ring retractions with nilpotent kernel.  We do not claim here that this level of generality is absent from all prior literature.

The paper is organized as follows.  Section~2 fixes conventions and elementary
facts about split nilpotent retractions.  Section~3 proves the main
change-of-rings formula.  Section~4 derives equality of flat projective-dimension
spectra and the general $n$-perfectness theorem.  Section~5 specializes to
commutative trivial extensions, gives the requested list of equivalent
characterizations, and proves cotorsion-dimension invariance.   Section~6 develops localization and iteration consequences.  Section~7
applies the general theorem to $n$-trivial extensions, truncated polynomial rings,
and triangular matrix rings.  Section~8 records examples and sharpness.

\section{Split-by-nilpotent extensions and preliminaries}

Unless commutativity is explicitly assumed, rings are associative with identity and
modules are unital left modules.  For a ring $A$ and an $A$-module $X$, the symbols
$\pd_AX$ and $\fd_AX$ denote projective and flat dimension, respectively, with
values in $\mathbb N\cup\{0,\infty\}$.

\begin{definition}\label{def:splitnil}
A ring $S$ is a \emph{split-by-nilpotent extension} of a ring $R$ if there are identity-preserving ring homomorphisms
\[
  R\xrightarrow{\ \iota\ }S\xrightarrow{\ \pi\ }R
\]
with $\pi\iota=\id_R$ and with $J:=\ker\pi$ nilpotent.  Thus $J^t=0$ for some
$t\ge1$.  Via $\iota$, we identify $R$ with a subring of $S$.
\end{definition}

As an $R$-$R$-bimodule one has $S\cong R\oplus J$, but this decomposition is not
asserted to split in the category of $S$-modules.  The quotient $R\cong S/J$ is
viewed as a right or left $S$-module through $\pi$ whenever tensor products or Tor
groups over $S$ are formed.

For a left $S$-module $X$, put
\[
   \overline X:=R\otimes_SX.
\]
There is a canonical isomorphism of left $R$-modules
\begin{equation}\label{eq:quotient}
   \overline X\cong X/JX.
\end{equation}

We repeatedly use the following nilpotent form of Nakayama's elementary argument.

\begin{lemma}[Nilpotent Nakayama argument]\label{lem:nil-nakayama}
Let $J$ be a nilpotent ideal of a ring $S$, and let $Y$ be an $S$-module.  If
$Y=JY$, then $Y=0$.
\end{lemma}

\begin{proof}
If $J^t=0$, then
\[
  Y=JY=J^2Y=\cdots=J^tY=0.
\]
No finite generation hypothesis on $Y$ is needed.
\end{proof}

\begin{lemma}[Induced flat modules]\label{lem:induced-flat}
Let $S$ be any ring containing a subring $R$ with the same identity.  If $C$ is a
flat left $R$-module, then $S\otimes_RC$ is flat as a left $S$-module.
If, in addition, $S\twoheadrightarrow R$ is a ring retraction of the inclusion,
then
\[
   R\otimes_S(S\otimes_RC)\cong C.
\]
If $C$ is projective over $R$, then $S\otimes_RC$ is projective over $S$.
\end{lemma}

\begin{proof}
For every right $S$-module $N$ there is a natural isomorphism
\[
 N\otimes_S(S\otimes_RC)\cong N\otimes_RC,
\]
where $N$ is restricted to a right $R$-module.  Since $C$ is $R$-flat, the functor
on the right is exact in $N$, proving $S$-flatness.  Under a retraction,
associativity gives
\[
 R\otimes_S(S\otimes_RC)
 \cong (R\otimes_SS)\otimes_RC
 \cong C.
\]
Finally, if $C$ is a direct summand of a free $R$-module, then $S\otimes_RC$ is a
direct summand of the corresponding free $S$-module.
\end{proof}

\begin{lemma}[Reduction of a flat module]\label{lem:reduction-flat}
Let $\pi:S\twoheadrightarrow R$ be any ring epimorphism and let $F$ be a flat left
$S$-module.  Then $R\otimes_SF$ is flat as a left $R$-module.
\end{lemma}

\begin{proof}
Let $0\to A'\to A\to A''\to0$ be an exact sequence of right $R$-modules and view
it as a sequence of right $S$-modules via $\pi$.  For every right $R$-module $A$
there is a natural isomorphism
\[
 A\otimes_R(R\otimes_SF)\cong A\otimes_SF.
\]
The right-hand functor is exact because $F$ is $S$-flat.  Hence
$R\otimes_SF$ is $R$-flat.
\end{proof}

\begin{definition}\label{def:torind}
Let $S\twoheadrightarrow R$ be a ring epimorphism.  A left $S$-module $X$ will be
called \emph{$R$-Tor-independent} if
\[
   \Tor_i^S(R,X)=0\qquad\text{for all }i>0.
\]
Every flat left $S$-module is $R$-Tor-independent.
\end{definition}

\section{The change-of-rings formula}

We now establish the main technical theorem.  Its proof is elementary once the
nilpotent retraction is used in the correct direction, but each step is needed: the
ring section provides induction $S\otimes_R-$, while nilpotence turns reduction
modulo $J$ into a detection principle for surjectivity and injectivity.

\begin{lemma}[Projective reduction detects projectivity]\label{lem:base-projective}
Let $S$ be a split-by-nilpotent extension of $R$, with kernel $J$, and let $X$ be an
$R$-Tor-independent left $S$-module.  If
\[
   C:=R\otimes_SX
\]
is projective over $R$, then $X$ is projective over $S$.  More precisely,
\[
   X\cong S\otimes_RC.
\]
\end{lemma}

\begin{proof}
Write $q:X\twoheadrightarrow X/JX\cong C$ for the quotient map.  This map is
$R$-linear via the fixed section $R\hookrightarrow S$.  Since $C$ is projective over
$R$, choose an $R$-linear section $\sigma:C\to X$ of $q$.  Define
\[
  \Phi:S\otimes_RC\longrightarrow X,
  \qquad s\otimes c\longmapsto s\sigma(c).
\]
It is $S$-linear.  After applying $R\otimes_S-$, the map $\Phi$ becomes the identity
of $C$: under the canonical identification
$R\otimes_S(S\otimes_RC)\cong C$, its reduction sends $c$ to $q\sigma(c)=c$.

Let $Q=\Coker\Phi$.  Right exactness of $R\otimes_S-$ gives
$R\otimes_SQ=0$, equivalently $Q/JQ=0$.  Hence $Q=JQ$, and
\cref{lem:nil-nakayama} gives $Q=0$.  Thus $\Phi$ is surjective.

Let $K=\Ker\Phi$.  From
\[
 0\longrightarrow K\longrightarrow S\otimes_RC
 \xrightarrow{\Phi}X\longrightarrow0
\]
and the hypothesis $\Tor_1^S(R,X)=0$, tensoring with $R$ yields an exact sequence
\[
 0\longrightarrow R\otimes_SK\longrightarrow C
 \xrightarrow{\id_C}C\longrightarrow0.
\]
Therefore $R\otimes_SK=0$, so $K=JK$.  Again
\cref{lem:nil-nakayama} gives $K=0$.  Hence $\Phi$ is an isomorphism.  Since $C$
is $R$-projective, \cref{lem:induced-flat} shows that $S\otimes_RC$ is
$S$-projective.
\end{proof}

\begin{theorem}[Tor-independent projective-dimension formula]\label{thm:tor-pd}
Let $S$ be a split-by-nilpotent extension of $R$.  If $X$ is an
$R$-Tor-independent left $S$-module, then
\begin{equation}\label{eq:tor-pd}
   {\ \pd_SX=\pd_R(R\otimes_SX).\ }
\end{equation}
The equality holds in $\mathbb N\cup\{0,\infty\}$.
\end{theorem}

\begin{proof}
Set $C=R\otimes_SX$.

\smallskip
\noindent\emph{Step 1: $\pd_SX\le\pd_RC$.}
If $\pd_RC=\infty$, there is nothing to prove.  Assume
$\pd_RC=n<\infty$ and argue by induction on $n$.

If $n=0$, the assertion is exactly \cref{lem:base-projective}.

Let $n>0$ and assume the result known in degree $n-1$.  Choose an exact sequence
of left $R$-modules
\begin{equation}\label{eq:Rstep}
 0\longrightarrow C'\longrightarrow P\xrightarrow{\varepsilon}C
 \longrightarrow0
\end{equation}
with $P$ projective and $\pd_RC'=n-1$ (the equality is the standard syzygy formula for a module of finite positive projective dimension).  Since $P$ is $R$-projective and
$q:X\twoheadrightarrow C$ is $R$-linear, there is an $R$-linear lift
$u:P\to X$ such that $qu=\varepsilon$.  By adjunction, $u$ induces an
$S$-linear map
\[
 \widetilde u:S\otimes_RP\longrightarrow X,
 \qquad s\otimes p\longmapsto su(p).
\]
Its reduction modulo $J$ is $\varepsilon:P\twoheadrightarrow C$.  If
$Q=\Coker\widetilde u$, then $R\otimes_SQ=0$, hence $Q=JQ$, and nilpotence of
$J$ gives $Q=0$.  Thus $\widetilde u$ is surjective.  Let
$K=\Ker\widetilde u$.

The exact sequence
\[
 0\longrightarrow K\longrightarrow S\otimes_RP
 \longrightarrow X\longrightarrow0
\]
shows, by the long exact Tor sequence and the projectivity of $S\otimes_RP$, that
\[
 \Tor_i^S(R,K)\cong\Tor_{i+1}^S(R,X)=0
 \qquad(i>0).
\]
Thus $K$ is $R$-Tor-independent.  Moreover, using
$\Tor_1^S(R,X)=0$, reduction gives an exact sequence
\[
 0\longrightarrow R\otimes_SK\longrightarrow P
 \xrightarrow{\varepsilon}C\longrightarrow0.
\]
Comparing with \eqref{eq:Rstep}, we obtain
$R\otimes_SK\cong C'$.  The induction hypothesis therefore yields
\[
  \pd_SK=\pd_RC'=n-1.
\]
Since $S\otimes_RP$ is projective over $S$, it follows that
$\pd_SX\le n$.

\smallskip
\noindent\emph{Step 2: $\pd_RC\le\pd_SX$.}
If $\pd_SX=\infty$, the inequality is automatic unless $\pd_RC<\infty$; but in
that case Step~1 would force $\pd_SX<\infty$, a contradiction.  Hence we may assume
$\pd_SX=m<\infty$.

Choose an $S$-projective resolution of length $m$,
\[
 0\longrightarrow P_m\longrightarrow P_{m-1}\longrightarrow\cdots
 \longrightarrow P_0\longrightarrow X\longrightarrow0.
\]
The homology of the complex obtained by applying $R\otimes_S-$ computes
$\Tor_i^S(R,X)$.  Since $X$ is $R$-Tor-independent, the reduced complex is exact:
\[
 0\longrightarrow R\otimes_SP_m\longrightarrow\cdots
 \longrightarrow R\otimes_SP_0\longrightarrow C\longrightarrow0.
\]
Each $R\otimes_SP_i$ is projective over $R$, because a split summand of an
$S$-free module remains a split summand of an $R$-free module after applying
$R\otimes_S-$.  Therefore $\pd_RC\le m=\pd_SX$.

Combining both steps proves \eqref{eq:tor-pd}.
\end{proof}

\begin{remark}\label{rem:compare-palmer}
For classical trivial extensions, projective-dimension formulae under Tor-vanishing
conditions are part of the homological framework developed by Palm\'er and Roos
\cite{PalmerRoos1973}.  The point of \cref{thm:tor-pd} is that the nilpotent
retraction itself is enough for the displayed equality once
$\Tor_{>0}^S(R,X)=0$ is assumed; no flatness or projectivity condition on the
kernel $J$ is required.
\end{remark}

Flat modules give the central special case.

\begin{corollary}[Flat projective-dimension formula]\label{cor:flat-pd}
Let $S$ be a split-by-nilpotent extension of $R$ with kernel $J$.  If $F$ is a flat
left $S$-module, then $F/JF$ is flat over $R$ and
\begin{equation}\label{eq:flat-pd}
   {\ \pd_SF=\pd_R(F/JF).\ }
\end{equation}
In particular, a flat $S$-module $F$ is projective if and only if $F/JF$ is
projective over $R$.
\end{corollary}

\begin{proof}
Flatness implies $\Tor_i^S(R,F)=0$ for $i>0$, while
\cref{lem:reduction-flat} shows that $F/JF\cong R\otimes_SF$ is $R$-flat.
Now apply \cref{thm:tor-pd}.  The last statement is the case of projective
dimension zero.
\end{proof}

\section{Flat projective-dimension spectra and \texorpdfstring{$n$}{n}-perfectness}

The modulewise formula yields more information than equality of suprema.

\begin{definition}\label{def:spectrum}
For a ring $A$, define its \emph{left flat projective-dimension spectrum} by
\[
 \FlatPD_{\ell}(A)
 :=\{\pd_AF\mid F\text{ is a flat left }A\text{-module}\}
 \subseteq\mathbb N\cup\{0,\infty\}.
\]
Set
\[
 \splf_{\ell}(A):=\sup\FlatPD_{\ell}(A).
\]
Thus $A$ is left $n$-perfect exactly when $\splf_{\ell}(A)\le n$.
\end{definition}

\begin{theorem}[Spectrum invariance]\label{thm:spectrum}
If $S$ is a split-by-nilpotent extension of $R$, then
\begin{equation}\label{eq:spectrum}
   {\ \FlatPD_{\ell}(S)=\FlatPD_{\ell}(R).\ }
\end{equation}
Consequently,
\[
   \splf_{\ell}(S)=\splf_{\ell}(R).
\]
\end{theorem}

\begin{proof}
Let $F$ be a flat left $S$-module.  By \cref{cor:flat-pd},
$C=R\otimes_SF$ is flat over $R$ and
$\pd_SF=\pd_RC$.  Hence
$\FlatPD_{\ell}(S)\subseteq\FlatPD_{\ell}(R)$.

Conversely, let $C$ be flat over $R$.  By \cref{lem:induced-flat},
$S\otimes_RC$ is flat over $S$, and its reduction is $C$.  Thus
\cref{cor:flat-pd} gives
\[
  \pd_S(S\otimes_RC)=\pd_RC.
\]
Therefore every value in $\FlatPD_{\ell}(R)$ occurs in
$\FlatPD_{\ell}(S)$, proving equality of the sets.  Taking suprema gives the last
formula.
\end{proof}

\begin{theorem}[$n$-perfectness under split nilpotent extensions]\label{thm:nperfect-general}
Let $n\ge0$.  If $S$ is a split-by-nilpotent extension of $R$, then
\[
   S\text{ is left }n\text{-perfect}
   \quad\Longleftrightarrow\quad
   R\text{ is left }n\text{-perfect}.
\]
The analogous equivalence holds for right $n$-perfectness.
\end{theorem}

\begin{proof}
The left-handed assertion follows immediately from
\cref{thm:spectrum}.  The right-handed proof is obtained by applying the same
argument to right modules; equivalently, apply the left-handed result to the
opposite-ring retraction $S^{\mathrm{op}}\twoheadrightarrow R^{\mathrm{op}}$.
\end{proof}

\begin{corollary}[Strict $n$-perfectness]\label{cor:strict}
Let $n\ge1$.  Under the hypotheses of \cref{thm:nperfect-general}, $S$ is left
$n$-perfect but not left $(n-1)$-perfect if and only if the same is true of $R$.
Also, $S$ is left perfect if and only if $R$ is left perfect.
\end{corollary}

\begin{proof}
Use \cref{thm:nperfect-general} for $n$ and $n-1$.  The final statement is the case
$n=0$.
\end{proof}

\begin{remark}\label{rem:bass}
For $n=0$, \cref{thm:nperfect-general} is compatible with Bass's structural theory
of perfect rings \cite{Bass1960}.  The present proof is homological and does not
require a separate analysis of Jacobson radicals or descending chains of principal
ideals.
\end{remark}

\section{Arbitrary trivial extensions: equivalent characterizations}

For idealizations, Boutghouchout--Kim--Mahdou considered $R\ltimes P$ with
$P$ projective and proved
\[
\cotD(R\ltimes P)=\cotD(R),
\]
obtaining corresponding consequences for perfectness and $n$-perfectness
\cite{BoutghouchoutKimMahdou}.  They also asked whether the equality remains
true for $R\ltimes M$ when $M$ is arbitrary.
 We will give an answer in this section.

 Throughout this
section, $R$ is a commutative ring, $M$ is an arbitrary $R$-module,
\[
   S=R\ltimes M,
   \qquad I=0\ltimes M.
\]
No condition is imposed on $M$.

For a commutative ring $A$ and an $A$-module $X$, recall that the cotorsion
dimension of $X$ is
\[
  \cd_A(X)
  =\inf\bigl\{n\ge0:\Ext_A^{n+1}(F,X)=0
  \text{ for every flat }A\text{-module }F\bigr\},
\]
with value $\infty$ if no such $n$ exists.  The global cotorsion dimension is
\[
  \cotD(A)=\sup\{\cd_A(X):X\in\Mod A\}.
\]

\begin{lemma}[Flat syzygies]\label{lem:flat-syzygy}
Let $A$ be a ring and let
\[
 0\longrightarrow K\longrightarrow P\longrightarrow F\longrightarrow0
\]
be exact, where $P$ is projective and $F$ is flat.  Then $K$ is flat.  Consequently,
all syzygies in a projective resolution of a flat module are flat.
\end{lemma}

\begin{proof}
For every right $A$-module $N$, the long exact Tor sequence gives
\[
 \Tor_1^A(N,K)\cong \Tor_2^A(N,F)=0,
\]
because $P$ is projective and $F$ is flat.  Hence $K$ is flat.  Iteration proves
the final assertion.
\end{proof}

\begin{proposition}[Cotorsion dimension and flat projective dimension]\label{prop:cot-splf}
For every commutative ring $A$,
\[
  \cotD(A)=\sup\{\pd_AF\mid F\text{ is flat over }A\}=\splf(A).
\]
Consequently, for $n\ge0$, $A$ is $n$-perfect if and only if
$\cotD(A)\le n$.
\end{proposition}

\begin{proof}
Put
\[
 d:=\sup\{\pd_AF\mid F\text{ flat}\}.
\]
We compare finite upper bounds.  If $d\le n$, then every flat $F$ has
$\pd_AF\le n$, hence $\Ext_A^{n+1}(F,X)=0$ for every $A$-module $X$; therefore
$\cotD(A)\le n$.

Conversely, assume $\cotD(A)\le n$.  Fix an $A$-module $X$.  By definition there
is some $m\le n$ such that
\[
 \Ext_A^{m+1}(F,X)=0
 \qquad\text{for every flat }F.
\]
By \cref{lem:flat-syzygy} and dimension shifting, the same vanishing holds in
every degree $j\ge m+1$, in particular in degree $n+1$.  Thus
\[
 \Ext_A^{n+1}(F,X)=0
\]
for every flat $F$ and every $X$.  The standard projective-dimension criterion
then gives $\pd_AF\le n$ for every flat $F$, so $d\le n$.

Hence $d$ and $\cotD(A)$ have exactly the same finite upper bounds.  This proves
the equality in $\mathbb N\cup\{0,\infty\}$.  Compare
\cite{MaoDing2006,BennisMahdou2009}.
\end{proof}

\begin{theorem}[Complete $n$-perfectness characterization for idealizations]\label{thm:equiv}
Let $R$ be a commutative ring, $M$ an arbitrary $R$-module,
$S=R\ltimes M$, and $I=0\ltimes M$.  For an integer $n\ge0$, the following are
equivalent:
\begin{enumerate}[label=\textup{(\roman*)}]
\item $R$ is $n$-perfect;
\item $S$ is $n$-perfect;
\item $\cotD(R)\le n$;
\item $\cotD(S)\le n$;
\item for every flat $S$-module $F$, the $R$-module $F/IF$ is flat and
      \[
        \pd_R(F/IF)\le n;
      \]
\item for every flat $S$-module $F$,
      \[
        \pd_SF=\pd_R(F/IF)\le n;
      \]
\item for every flat $R$-module $C$, the induced $S$-module $S\otimes_RC$ is flat
      and
      \[
        \pd_S(S\otimes_RC)=\pd_RC\le n;
      \]
\item $\FlatPD(S)=\FlatPD(R)\subseteq\{0,1,\dots,n\}$.
\end{enumerate}
\end{theorem}

\begin{proof}
The idealization projection $S\twoheadrightarrow R$ is split and has square-zero
kernel $I$, hence it is a split-by-nilpotent extension.  The equivalence of
(i), (ii), (v), (vi), (vii), and (viii) follows from
\cref{cor:flat-pd,thm:spectrum,thm:nperfect-general}.  The equivalence of
(i) with (iii), and of (ii) with (iv), is \cref{prop:cot-splf}.
\end{proof}

The global equality is stronger than the requested finite bound.

\begin{theorem}[Cotorsion-dimension invariance for arbitrary idealizations]\label{thm:cotD}
For every commutative ring $R$ and every $R$-module $M$,
\begin{equation}\label{eq:cotD}
   {\ \cotD(R\ltimes M)=\cotD(R).\ }
\end{equation}
Moreover,
\begin{equation}\label{eq:spec-trivial}
   \FlatPD(R\ltimes M)=\FlatPD(R).
\end{equation}
\end{theorem}

\begin{proof}
Equation \eqref{eq:spec-trivial} is \cref{thm:spectrum}.  Taking suprema and using
\cref{prop:cot-splf} gives \eqref{eq:cotD}.
\end{proof}

\begin{corollary}\label{cor:perfect-arbitrary}
Let $M$ be any $R$-module.  Then $R\ltimes M$ is perfect if and only if $R$ is
perfect.  More generally, if $\cotD(R)=n<\infty$, then $R\ltimes M$ is exactly
$n$-perfect in the sense that it is $n$-perfect and, for $n>0$, it is not
$(n-1)$-perfect.
\end{corollary}

\begin{proof}
Use \cref{thm:equiv,thm:cotD}.
\end{proof}

\begin{proposition}[Restriction of cotorsion dimension]\label{prop:cd-restrict}
Let $N$ be an $S=R\ltimes M$-module.  Then
\[
   \cd_R( N)\le\cd_S(N).
\]
In particular, every cotorsion $S$-module is cotorsion after restriction of
scalars to $R$.
\end{proposition}

\begin{proof}
Let $C$ be a flat $R$-module.  Choose an $R$-projective resolution
$P_\bullet\to C$.  Since $C$ is flat,
$\Tor_i^R(S,C)=0$ for $i>0$, so $S\otimes_RP_\bullet$ is an exact
$S$-projective resolution of $S\otimes_RC$.  Termwise adjunction gives
\[
 \operatorname{Hom}_S(S\otimes_RP_j,N)
 \cong\operatorname{Hom}_R(P_j,N),
\]
and hence
\begin{equation}\label{eq:Extind}
 \Ext_S^i(S\otimes_RC,N)
 \cong \Ext_R^i(C,N)
 \qquad(i\ge0).
\end{equation}
By \cref{lem:induced-flat}, $S\otimes_RC$ is $S$-flat.  Therefore a bound on
$\cd_SN$ gives the same bound on $\cd_R(N)$.
\end{proof}

\begin{remark}\label{rem:coinduction-warning}
The converse transport of cotorsion through coinduction requires separate care.
The relevant functor is
\[
  \operatorname{Hom}_R(S,-).
\]
If $M$ is arbitrary, $S\cong R\oplus M$ need not be projective over $R$, so
this functor need not be exact.  No such exactness is used in
\cref{thm:cotD}.
\end{remark}

\section{Localization and arbitrary iteration}

We next record consequences that are safe without invoking any unqualified global
local--global formula for cotorsion dimension.

\begin{proposition}[Pointwise localization]\label{prop:local}
Let $R$ be commutative, $M$ an $R$-module, $S=R\ltimes M$, and
$\mathfrak p\in\Spec R$.  Put
\[
  \mathfrak P=\mathfrak p\ltimes M.
\]
Then every prime ideal of $S$ is uniquely of this form and
\[
  S_{\mathfrak P}\cong R_{\mathfrak p}\ltimes M_{\mathfrak p}.
\]
Consequently, for every $n\ge0$,
\[
  S_{\mathfrak P}\text{ is }n\text{-perfect}
  \quad\Longleftrightarrow\quad
  R_{\mathfrak p}\text{ is }n\text{-perfect},
\]
and
\[
  \cotD(S_{\mathfrak P})=\cotD(R_{\mathfrak p}).
\]
\end{proposition}

\begin{proof}
The square-zero ideal $I=0\ltimes M$ is contained in every prime ideal of $S$.
Thus $\Spec S\to\Spec(S/I)\cong\Spec R$ is a bijection, with
$\mathfrak p$ corresponding to $\mathfrak P$.  Localization commutes with the
idealization construction, giving the displayed ring isomorphism.  Apply
\cref{thm:equiv,thm:cotD} over $R_{\mathfrak p}$.
\end{proof}

\begin{corollary}[Locally $n$-perfect]\label{cor:locally}
With the notation above, the following are equivalent:
\begin{enumerate}[label=\textup{(\roman*)}]
\item $R_{\mathfrak p}$ is $n$-perfect for every $\mathfrak p\in\Spec R$;
\item $S_{\mathfrak P}$ is $n$-perfect for every $\mathfrak P\in\Spec S$.
\end{enumerate}
\end{corollary}

\begin{proof}
This is the pointwise statement of \cref{prop:local} under the natural bijection of
prime spectra.
\end{proof}

\begin{remark}\label{rem:localglobal-warning}
\Cref{prop:local,cor:locally} are deliberately pointwise.  They do not require, and
we do not use, a formula asserting that the global cotorsion dimension of an
arbitrary commutative ring equals the supremum of the cotorsion dimensions of all
its localizations.
\end{remark}

\begin{theorem}[Arbitrary iterated idealizations]\label{thm:iteration}
Let $R$ be commutative, put $R_0=R$, and define recursively
\[
  R_{i+1}=R_i\ltimes M_i\qquad(i\ge0),
\]
where each $M_i$ is an arbitrary $R_i$-module.  Then every $R_i$ is commutative, and for every $i$
\[
  \FlatPD(R_i)=\FlatPD(R),
  \qquad
  \cotD(R_i)=\cotD(R).
\]
Hence $R_i$ is $n$-perfect if and only if $R$ is $n$-perfect.
\end{theorem}

\begin{proof}
Apply \cref{thm:spectrum,thm:cotD} at each stage and argue by induction.
\end{proof}

\section{Beyond ordinary idealizations}

The main theorem applies to every split nilpotent retraction, so several familiar
constructions are immediate consequences.

\subsection{\texorpdfstring{$n$}{n}-trivial extensions}

Let $R\ltimes_n(M_1,\dots,M_n)$ be an $n$-trivial extension in the sense of
\cite{AndersonEtAl2017,BenkhadraBennisGarcia2020}.  Additively it is
$R\oplus M_1\oplus\cdots\oplus M_n$, the projection onto $R$ is a ring
retraction, and the positive part
\[
  J=M_1\oplus\cdots\oplus M_n
\]
is nilpotent, with $J^{n+1}=0$.

\begin{corollary}[$n$-trivial extensions]\label{cor:ntrivial}
Let $T=R\ltimes_n(M_1,\dots,M_n)$.  If $F$ is a flat left $T$-module, then
\[
  \pd_TF=\pd_R(R\otimes_TF).
\]
Moreover,
\[
  \FlatPD_{\ell}(T)=\FlatPD_{\ell}(R),
\]
and $T$ is left $k$-perfect if and only if $R$ is left $k$-perfect, for every
$k\ge0$.  The same statements hold on the right.
\end{corollary}

\begin{proof}
Apply \cref{cor:flat-pd,thm:spectrum,thm:nperfect-general} to the canonical split
nilpotent retraction $T\twoheadrightarrow R$.
\end{proof}

\begin{remark}\label{rem:ntrivial-literature}
The $k$-perfect equivalence in \cref{cor:ntrivial} recovers
\cite[Theorem~4.2]{BenkhadraBennisGarcia2020}.  The modulewise equality of
projective dimensions and the equality of the complete spectra
$\FlatPD_{\ell}$ are stronger statements not needed in the categorical proof of
that equivalence.
\end{remark}

\subsection{Truncated polynomial rings}

\begin{corollary}[Truncated polynomial extensions]\label{cor:truncated}
Let $R$ be a ring and let $t\ge2$.  Suppose $x$ is central over $R$ and set
\[
   S=R[x]/(x^t).
\]
Then the evaluation map $x\mapsto0$ makes $S$ a split-by-nilpotent extension of
$R$.  Consequently,
\[
  \FlatPD_{\ell}(S)=\FlatPD_{\ell}(R)
\]
and $S$ is left $n$-perfect if and only if $R$ is left $n$-perfect.  In the
commutative case,
\[
  \cotD(R[x]/(x^t))=\cotD(R).
\]
\end{corollary}

\begin{proof}
The kernel $(x)/(x^t)$ is nilpotent and the constant-polynomial embedding is a ring
section.  Apply the general results.
\end{proof}

\subsection{Upper triangular matrix rings}

Let $T_m(R)$ denote the ring of $m\times m$ upper triangular matrices over $R$.
Projection onto the diagonal gives a split epimorphism
\[
  T_m(R)\twoheadrightarrow R^m
\]
whose kernel is the nilpotent ideal of strictly upper triangular matrices.

\begin{corollary}[Upper triangular matrix rings]\label{cor:triangular}
For every ring $R$ and $m\ge1$,
\[
  \FlatPD_{\ell}(T_m(R))=\FlatPD_{\ell}(R^m)=\FlatPD_{\ell}(R).
\]
Since a finite product is left $n$-perfect exactly when each factor is left
$n$-perfect,
\[
  T_m(R)\text{ is left }n\text{-perfect}
  \quad\Longleftrightarrow\quad
  R\text{ is left }n\text{-perfect}.
\]
The right-handed analogue also holds.
\end{corollary}

\begin{proof}
Apply \cref{thm:spectrum,thm:nperfect-general} to the diagonal retraction.  A left $R^m$-module is a tuple $(X_1,\dots,X_m)$, flatness is componentwise, and
\[
 \pd_{R^m}(X_1,\dots,X_m)=\max_i\pd_R X_i.
\]
The maximum of finitely many elements of the totally ordered set $\mathbb N\cup\{0,\infty\}$ is one of them, and every value in $\FlatPD_\ell(R)$ is realized by the constant tuple.  Hence $\FlatPD_\ell(R^m)=\FlatPD_\ell(R)$, and the asserted $n$-perfect equivalence follows.
\end{proof}

\begin{remark}\label{rem:gldim}
\Cref{cor:triangular} concerns projective dimensions of flat modules, not global
dimension.  The global dimension of $T_m(R)$ can be strictly larger than that of
$R$.  Thus $n$-perfectness is substantially more stable under nilpotent extensions
than global dimension.
\end{remark}

\section{Examples and sharpness}

\begin{example}[A nonprojective idealizing module]\label{ex:zn}
Let $R=\mathbb Z$ and $M=\mathbb Z/n\mathbb Z$ with $n\ge2$.  The module $M$ is
not projective and has flat dimension one.  Since $\mathbb Z$ is $1$-perfect but
not perfect,
\[
  \cotD(\mathbb Z\ltimes\mathbb Z/n\mathbb Z)=1.
\]
Thus the projectivity hypothesis on the idealizing module is unnecessary even in
the simplest torsion example.
\end{example}

\begin{example}[An idealizing module of infinite flat dimension]\label{ex:arbitraryM}
Let $k$ be a field and
\[
 R=k[\varepsilon]/(\varepsilon^2),
 \qquad M=R/(\varepsilon)\cong k.
\]
The ring $R$ is Artinian, hence perfect.  On the other hand, $M$ has infinite flat
(and projective) dimension.  Indeed, multiplication by $\varepsilon$ gives an
infinite exact free resolution
\[
 \cdots\xrightarrow{\varepsilon}R
 \xrightarrow{\varepsilon}R
 \xrightarrow{\varepsilon}R\longrightarrow k\longrightarrow0.
\]
After tensoring with $k$, every displayed differential becomes zero.  Hence
\[
  \Tor_i^R(k,k)\cong k\quad(i\ge1),
  \qquad \fd_RM=\infty.
\]
Nevertheless, \cref{cor:perfect-arbitrary} gives
\[
 R\ltimes M\quad\text{perfect}.
\]
Thus finite flat dimension of the idealizing module is genuinely unnecessary.
\end{example}

\begin{example}[Arbitrarily large finite flat dimension]\label{ex:poly}
Let $d\ge1$, let $k$ be a field,
\[
  R=k[x_1,\dots,x_d],
  \qquad
  M=R/(x_1,\dots,x_d)\cong k.
\]
The Koszul complex gives $\pd_RM=\fd_RM=d$.  By the polynomial-ring computation
of Bennis and Mahdou \cite[Theorem~3.1]{BennisMahdou2009},
$\cotD(R)=d$.  Therefore
\[
  \cotD(R\ltimes k)=d.
\]
The integer controlling the flat dimension of $M$ can be as large as desired,
while the idealization retains exactly the same cotorsion dimension as $R$.
\end{example}

\begin{example}[Infinite-dimensional square-zero radical]\label{ex:field}
Let $k$ be a field and $V$ an arbitrary $k$-vector space, with no restriction on
its dimension.  Then
\[
  k\ltimes V
\]
is perfect by \cref{cor:perfect-arbitrary}.  When $V$ is infinite-dimensional, the ideal $0\ltimes V$ is not finitely generated, so the ring is non-Noetherian.  This illustrates that perfectness of the
idealization does not rely on Noetherian hypotheses.
\end{example}

\begin{example}[Truncated polynomials over $\mathbb Z$]\label{ex:truncZ}
For every $t\ge2$,
\[
  \mathbb Z[x]/(x^t)
\]
is $1$-perfect but not perfect, because the same is true of $\mathbb Z$ and
\cref{cor:truncated} preserves the exact flat projective-dimension spectrum.
\end{example}

\begin{example}[Triangular matrices]\label{ex:triZ}
For $m\ge2$, the upper triangular matrix ring $T_m(\mathbb Z)$ has global
homological behavior different from $\mathbb Z$, but
\[
  \splf_{\ell}(T_m(\mathbb Z))=1.
\]
Thus it is left $1$-perfect and not left perfect.  The same conclusion holds on the
right.
\end{example}

\begin{example}[Iterated arbitrary idealizations]\label{ex:iter}
Start with a commutative ring $R$ satisfying $\cotD(R)=n$ and form
\[
  R_1=R\ltimes M_0,
  \quad
  R_2=R_1\ltimes M_1,
  \quad\dots,
\]
with completely arbitrary modules $M_i$.  Then every $R_i$ has global cotorsion
dimension $n$.  This produces long families of generally non-isomorphic rings with
identical flat projective-dimension spectra.
\end{example}

\end{document}